\documentclass[12pt]{amsart}

\usepackage{latexsym}
\usepackage{psfrag}
\usepackage{amsmath}
\usepackage{amssymb}
\usepackage{epsfig}
\usepackage{amsfonts}
\usepackage{amscd}
\usepackage{mathrsfs}
\usepackage{graphicx}
\usepackage{enumerate}
\usepackage[autostyle=false, style=english]{csquotes}
\MakeOuterQuote{"}
\usepackage{ragged2e}
\usepackage[all]{xy}
\usepackage{mathtools}
\newlength\ubwidth

\usepackage[dvipsnames]{xcolor}

\usepackage{placeins}

\newtheorem{theorem}{Theorem}
\newtheorem{lemma}[theorem]{Lemma}

\newtheorem{cor}[theorem]{Corollary}

\newtheorem{defn}[theorem]{Definition}

\begin{document}

\title[Mod $d$ extensions of parity in overpartitions]
  {Modulo $d$ extension of parity results in  
  Rogers-Ramanujan-Gordon type overpartition identities}

\author[Kur\c{s}ung\"{o}z]{Ka\u{g}an Kur\c{s}ung\"{o}z}
\address{Ka\u{g}an Kur\c{s}ung\"{o}z, 
  Faculty of Engineering and Natural Sciences, 
  Sabanc{\i} University, Tuzla, Istanbul 34956, Turkey}
\email{kursungoz@sabanciuniv.edu}

\author[Zadehdabbagh]{Mohammad Zadehdabbagh}
\address{Mohammad Zadehdabbagh, 
  Faculty of Engineering and Natural Sciences, 
  Sabanc{\i} University, Tuzla, Istanbul 34956, Turkey}
\email{mzadehdabbagh@sabanciuniv.edu}




\date{October 2022}

\begin{abstract}
\noindent 
Sang, Shi and Yee, in 2020, found overpartition analogs of 
Andrews' results involving parity in Rogers-Ramanujan-Gordon identities.  
Their result partially answered an open question of Andrews'.  
The open question was to involve parity in overpartition identities.  
We extend Sang, Shi, and Yee's work to arbitrary moduli, 
and also provide a missing case in their identities.  
We also unify proofs of Rogers-Ramanujan-Gordon identities 
for overpartitions due to Lovejoy and Chen et.al.; 
Sang, Shi, and Yee's results; and ours.  
Although verification type proofs are given for brevity, 
a construction of series as solutions of functional equations 
between partition generating functions is sketched.  
\end{abstract}

\maketitle

\section{Introduction and Statement of Results}
\label{secIntro}

Rogers-Ramanujan identities~\cite{RR, Rogers, Schur}
and their generalizations 
(\cite{Andrews-parity, Bressoud-RRG-AllModuli, ChenEtAl-RRG-overptn, RRG, 
KK-AndrewsStyle, Lovejoy-RRG-overptn, SSY-parity-overptn}, to exemplify a few)
constitute a significant part of all partition identities.  

\begin{theorem}[The first Rogers-Ramanujan identity~\cite{RR, Rogers, Schur}]
\label{thmRR1} 
  Let $n$ be a non-negative integer.  
  The number of partitions of $n$ into distinct and non-consecutive parts 
  equals the number of partitions of $n$ into parts 
  that are congruent to 1 or 4 modulo 5.  
\end{theorem}

A partition of a non-negative integer is a non-increasing sum of positive integers.  
The only partition of zero is agreed to be the empty partition~\cite{TheBlueBook}.  

For example, 9 has five partitions into distinct and non-consecutive parts: 
\begin{align}
\nonumber 
  9, \; 8 + 1, \; 7 + 2, \; 6 + 3, \; 5 + 3 + 1; 
\end{align}
as well as five partitions into parts 
that are congruent to 1 or 4 modulo 5: 
\begin{align}
\nonumber 
  9, \; 6 + 1 + 1 + 1, \; 4 + 4 + 1, \; 4 + 1 + \cdots + 1, \; 1 + \cdots + 1.  
\end{align}
The former condition is called a \emph{multiplicity condition} 
as it puts restrictions on how many times a part can appear, 
or a \emph{gap condition} as gaps are imposed between successive parts in a partition.  
The latter condition is called a \emph{congruence condition} 
for obvious reasons.  

The subsequent results are more compactly presented using the \emph{frequency notation}.  
For arbitrary but fixed partition $\lambda$ and positive integer $i$, 
$f_i$ is the number of times $i$ appears in $\lambda$~\cite{TheBlueBook}.  
For example, Taking $\lambda$ as $4+4+1$, $f_1 = 1$, $f_2 = f_3 = 0$, $f_4 = 2$ and $f_i = 0$ for $i \geq 5$.  

\begin{theorem}[Rogers-Ramanujan-Gordon identities~\cite{RRG}]
\label{thmRRG}
  Let $k$ and $a$ be positive integers satisfying $k \geq 2$ and $1 \leq a \leq k$.  
  Given any non-negative integer $n$, 
  let $A_{k, a}(n)$ be the number of partitions of $n$ into parts 
  that are not congruent to $0, \pm a \pmod{2k+1}$, 
  and $B_{k, a}(n)$ be the number of partitions of $n$ in which 
  $f_1 < a$ and $f_i + f_{i+1} < k$ for any $i$.  
  Then, $A_{k, a}(n) = B_{k, a}(n)$.  
\end{theorem}
$(k, a) = (2,2)$ and $=(2,1)$ are the first and the second Rogers-Ramanujan identities, 
respectively.  

The series 
\begin{align}
\nonumber 
  \sum_{n \geq 0} A_{k, a}(n) q^n, \quad 
  \textrm{ or } \quad 
  \sum_{n \geq 0} B_{k, a}(n) q^n \quad 
\end{align}
are called partition generating functions~\cite{TheBlueBook}.  
The study of partition identities goes hand in hand with $q$-series.  
The technique that is most often used in proving a Rogers-Ramanujan generalization 
involves a variant of the series 
\begin{align}
\nonumber
  Q_{k, a}(x; q) = \sum_{n \geq 0}
  \frac{ (-1)^n x^{kn} q^{ (2k+1)\binom{n+1}{2} - an} }{ (q; q)_n (xq^{n+1}; q)_\infty }
  - \frac{ (-1)^n x^{kn + a} q^{ (2k+1)\binom{n+1}{2} + a(n+1)} }
    { (q; q)_n (xq^{n+1}; q)_\infty }
\end{align}
due to Andrews~\cite{Andrews-RRG-analytic}.  
The origin of these series is in Selberg's work~\cite{Selberg}.  
For instance, Andrews~\cite{Andrews-RRG-analytic, Andrews-posets-RRG} showed that 
\begin{align}
\nonumber 
  \sum_{m, n \geq 0} b_{k, a}(m, n) x^m q^n = Q_{k, a}(x; q), 
\end{align}
where $b_{k, a}(m, n)$ are the number of partitions 
counted by $B_{k, a}(n)$ having $m$ parts.  
It is a simple matter to show that~\cite{TheBlueBook}
\begin{align}
\nonumber 
  \sum_{n \geq 0} A_{k, a}(n)q^n 
  = \prod_{ \substack{ n \geq 1 \\ n \not\equiv 0, \pm a \!\!\!\! \pmod{2k+1}} } 
    \frac{ 1 }{ (1 - q^n) }
  = \frac{ (q^a, q^{2k+1-a}, q^{2k+1}; q^{2k+1}) }{ (q; q)_\infty }.  
\end{align}
Here and elsewhere~\cite{GR}, 
\begin{align}
\nonumber 
  (a; q)_n = \prod_{j = 1}^n (1 - aq^{j-1}), \quad 
  (a_1, \ldots, a_m; q)_n = (a_1; q)_n \cdots (a_m; q)_n, \quad 
  (a; q)_\infty = \lim_{n \to \infty} (a; q)_n, 
\end{align}
and any series and infinite products in this paper converge absolutely for $\vert q \vert < 1$.  
When the base is $q = q^1$, we may abbreviate further as $(a)_n = (a; q)_n$.  

By a Rogers-Ramanujan generalization 
we mean a partition identity with a congruence condition on the right-hand side 
and a gap condition pertaining to only two consecutive parts on the left-hand side.  
Schur's partition identity~\cite{Schur} 
is not a Rogers-Ramanujan generalization in this sense, 
because the gap conditions are $f_i + f_{i+1} + f_{i+2} < 2$ and $f_{3i} + f_{3i+3} < 2$, 
involving parts that are up to three apart.  
We must impress that this is not a widely adopted classification.  
Many fellow researchers consider Schur's partition identity 
as a Rogers-Ramanujan type identity, 
because it relates gap conditions to congruence conditions.  

After Corteel and Lovejoy defined overpartitions~\cite{CL-overptns}, 
Lovejoy~\cite{Lovejoy-RRG-overptn} and Chen et.al.~\cite{ChenEtAl-RRG-overptn} 
demonstrated Gordon's theorem for overpartitions.  
An overpartition is a partition in which the first occurrence of any part may be overlined.  
For instance, 
\begin{align}
\nonumber 
  4 + 4 + 1, \quad
  \overline{4} + 4 + 1, \quad
  4 + 4 + \overline{1}, \quad
  \textrm{ and }
  \overline{4} + 4 + \overline{1}
\end{align}
are all of the overpartitions with underlying partition $4+4+1$.  
Frequency notation can be extended to overpartitions.  
For arbitrary but fixed overpartition $\lambda$ and positive integer $i$, 
$f_i$ is the number of occurrences of the non-overlined $i$'s in $\lambda$, 
and $f_{\overline{i}}$ is the number of occurrences of the overlined $i$ in $\lambda$.  
The definition implies that $f_{\overline{i}}$ can only be 0 or 1.  
For example, taking $\lambda$ to be $4 + 4 + \overline{1}$, 
$f_i = $ 2 if $i = 4$ and 0 otherwise, 
and $f_{\overline{i}} = $ 1 if $i = 1$ and 0 otherwise.  

\begin{theorem}[Gordon's Theorem for 
overpartitions~\cite{ChenEtAl-RRG-overptn, Lovejoy-RRG-overptn}]
\label{thmRRG-overptn}
  Let $k$ and $a$ be positive integers satisfying $k \geq 2$ and $1 \leq a \leq k$.  
  Given any non-negative integer $n$, 
  let $C_{k, a}(n)$ be the number of overpartitions of $n$ such that 
  $f_1 < a$ and $f_i + f_{\overline{i}} + f_{i+1} < k$ for all $i$.  
  For any $n$ and $1 \leq a < k$, 
  let $D_{k, a}(n)$ be the number of overpartitions of $n$ such that 
  the non-overlined parts are not congruent to 0 or $\pm a \pmod{2k}$, 
  and $D_{k,k}(n)$ be the number of overpartitions of $n$ such that 
  no parts are divisible by $k$.  
  Then, $C_{k, a}(n) = D_{k, a}(n)$.  
\end{theorem}

In 2010, Andrews studied parity in partition identities~\cite{Andrews-parity}.  
He gave a few Rogers-Ramanujan generalizations involving parity, 
one of which is the following.  
\begin{theorem}[{~\cite[Theorem~1]{Andrews-parity}}]
\label{thmRRG-GEA-parity}
  Suppose $k$, $a$ are integers satisfying $1 \leq k \leq a$ and $k \equiv a \pmod{2}$.  
  Let $\mathcal{W}_{k, a}(n)$ be the number of partitions 
  enumerated by $B_{k, a}(n)$ with the added restriction that 
  \emph{even} numbers appear an \emph{even} number of times.  
  If $k$ and $a$ are both even, let $G_{k, a}(n)$ be the number of partitions of $n$ 
  in which no odd part is repeated 
  and no even part is congruent to 0, $\pm a \pmod{2k+2}$.  
  If $k$ and $a$ are both odd, let $G_{k, a}(n)$ be the number of partitions of $n$ 
  into parts that are neither congruent to 2 $\pmod{4}$ 
  nor congruent to 0, $\pm a \pmod{2k+2}$.  
  Then, for all $n$, $\mathcal{W}_{k, a}(n) = G_{k, a}(n)$.  
\end{theorem}

One of the open problems Andrews listed at the end of~\cite{Andrews-parity}
was the extension of the results therein to overpartitions.  
Sang, Shi and Yee found a Rogers-Ramanujan generalization for overpartitions 
involving parity conditions~\cite{SSY-parity-overptn}.  
Their gap conditions are central in this note.  

\begin{defn}[{~\cite[Definitions 1.10, 1.11]{SSY-parity-overptn}}]
\label{defParityOverptn}
  For integers $k$ and $a$ satisfying $1 \leq a \leq k$, 
  let $U_{k, a}(n)$ be the number of overpartitions of $n$ satisfying 
  \begin{itemize}
  \item $f_1 \leq a-1 + f_{\overline{1}}$, 
  \item $f_{2l-1} \geq f_{\overline{2l-1}}$, 
  \item $f_{2l} + f_{\overline{2l}} \equiv 0 \pmod{2}$, 
  \item $f_l + f_{\overline{l}} + f_{l+1} \leq k - 1 + f_{\overline{l+1}}$; 
  \end{itemize}
  and let $\overline{U}_{k,a}(n)$ be the number of overpartitions of $n$ satisfying
  \begin{itemize}
  \item $f_1 \leq a-1 + f_{\overline{1}}$, 
  \item $f_{2l} \geq f_{\overline{2l}}$, 
  \item $f_{2l-1} + f_{\overline{2l-1}} \equiv 0 \pmod{2}$, 
  \item $f_l + f_{\overline{l}} + f_{l+1} \leq k - 1 + f_{\overline{l+1}}$.  
  \end{itemize}
\end{defn}

\begin{theorem}[{~\cite[Theorems 1.12, 1.13]{SSY-parity-overptn}}]
\label{thmSangShiYee}
  \begin{align}
  \nonumber 
    \sum_{n \geq 0} U_{2k, 2a}(n) q^n 
    = & \frac{ (-q; q)_\infty ( q^{2a}, q^{4k-2a}, q^{4k}; q^{4k} )_\infty }
      { (q^2; q^2)_\infty } 
  \end{align}
  \begin{align}
  \nonumber
    \sum_{n \geq 0} U_{2k, 2a-1}(n) q^n 
    = & \frac{1}{(1+q)} \frac{ (-q; q)_\infty ( q^{2a}, q^{4k-2a}, q^{4k}; q^{4k} )_\infty }
      { (q^2; q^2)_\infty } \\
  \nonumber 
    & + \frac{q}{(1+q)} \frac{ (-q; q)_\infty ( q^{2a-2}, q^{4k-2a+2}, q^{4k}; q^{4k} )_\infty }
      { (q^2; q^2)_\infty } 
  \end{align}
  \begin{align}
  \nonumber 
    \sum_{n \geq 0} \overline{U}_{2k, 2a-1}(n) q^n 
    = \sum_{n \geq 0} \overline{U}_{2k, 2a}(n) q^n 
    = & \frac{ (-q^2; q^2)^2_\infty ( q^{2a}, q^{4k-2a}, q^{4k}; q^{4k} )_\infty }
      { (q^2; q^2)_\infty } 
  \end{align}
\end{theorem}
Sang, Shi and Yee used Bailey pairs~\cite{Andrews-bailey} to prove the theorem above.  

Our main result is the modulo $d$ extension of Theorem \ref{thmSangShiYee}, 
as well as a constructive and unifying approach to the proofs.  
This will also supply the missing cases of Theorem \ref{thmSangShiYee}.  

We will be using the non-negative parameters $d$, $k$, $a$, $e$, and $f$ such that 
\begin{align}
\label{defParams}
  d \geq 1, \quad k \geq 1, \quad 
  0 \leq a \leq k, \quad \textrm{ and } 1 \leq e, f \leq d.  
\end{align}
throughout.  
For extreme cases, we will allow $f = 0$, as explained in the proofs below.  
The main definition of this paper inspired by Definition \ref{defParityOverptn} 
is presented next.  

\begin{defn}
\label{defMain}
  For parameters described by \eqref{defParams} and any positive integer $l$, 
  let $u_{dk+e, da+f}(m, n)$ be the number of partitions of $n$ into $m$ parts satisfying 
  \begin{itemize}
  \item[{\bf (i)}] $ f_1 \leq da + f - 1 + (d-1)f_{\overline{1}} $, 
  \item[{\bf (ii)}] $ f_{2l-1} \geq (d-1)f_{\overline{2l-1}}$, 
  \item[{\bf (iii)}] $ f_{2l} + f_{\overline{2l}} \equiv 0 \pmod{d} $, 
  \item[{\bf (iv)}] $ f_l + f_{\overline{l}} + f_{l+1} 
    \leq dk + e - 1 + (d-1) f_{\overline{l+1}} $; 
  \end{itemize}
  and let $\overline{u}_{dk+e, da+f}(m, n)$ be 
  the number of partitions of $n$ into $m$ parts satisfying 
  \begin{itemize}
  \item[$\mathbf{\overline{(i)}}$] $ f_1 \leq da + f - 1 + (d-1)f_{\overline{1}} $, 
  \item[$\mathbf{\overline{(ii)}}$] $ f_{2l} \geq (d-1)f_{\overline{2l}}$, 
  \item[$\mathbf{\overline{(iii)}}$] $ f_{2l-1} + f_{\overline{2l-1}} \equiv 0 \pmod{d} $, 
  \item[$\mathbf{\overline{(iv)}}$] $ f_l + f_{\overline{l}} + f_{l+1} 
    \leq dk + e - 1 + (d-1) f_{\overline{l+1}} $.  
  \end{itemize}
  Set 
  \begin{align}
  \nonumber 
    U_{dk+e, da+f}(x) = \sum_{m, n \geq 0} u_{dk+e, da+f}(m, n) x^m q^n, 
  \end{align}
  and 
  \begin{align}
  \nonumber 
    \overline{U}_{dk+e, da+f}(x) 
    = \sum_{m, n \geq 0} \overline{u}_{dk+e, da+f}(m, n) x^m q^n.  
  \end{align}
\end{defn}
The dependence of the enumerants and generating functions on $d$ is not made explicit 
to avoid an excess of indices such as ${}_du_{dk+e, da+f}(m, n)$.  
Then, we have the following theorem.  

\begin{theorem}
\label{thmOverPtnModd}
  For parameters as in \eqref{defParams} with $e = d$ or $2e = d$, 
  \begin{align}
  \nonumber 
    U_{dk+e, da+f}(1) = \begin{cases}
      \frac{ (1 - q^{d+f-e}) }{ (1 - q^d) }
      \frac{( q^{da+e}, q^{2dk-da+e}, q^{2dk+2e} ; q^{2dk+2e})_\infty}
      { ( q; q^{2})_\infty ( q^{d}; q^{d})_\infty }& \vspace{3mm} \\ 
      + \frac{ (q^{d+f-e} - q^d) }{ (1 - q^d) }
      \frac{( q^{da-d+e}, q^{2dk-da+d+e}, q^{2dk+2e} ; q^{2dk+2e})_\infty}
      { ( q; q^{2})_\infty ( q^{d}; q^{d})_\infty }, 
      & \textrm{ if } f < e, \vspace{5mm} \\
      \frac{( q^{da+e}, q^{2dk-da+e}, q^{2dk+2e} ; q^{2dk+2e})_\infty}
      { ( q; q^{2})_\infty ( q^{d}; q^{d})_\infty }, 
      & \textrm{ if } f = e, \vspace{5mm} \\
      \frac{ (q^{f-e} - q^d) }{ (1 - q^d) }
      \frac{( q^{da+e}, q^{2dk-da+e}, q^{2dk+2e} ; q^{2dk+2e})_\infty}
      { ( q; q^{2})_\infty ( q^{d}; q^{d})_\infty }& \vspace{3mm} \\
      + \frac{ (1 - q^{f-e}) }{ (1 - q^d) }
      \frac{( q^{da+d+e}, q^{2dk-da-d+e}, q^{2dk+2e} ; q^{2dk+2e})_\infty}
      { ( q; q^{2})_\infty ( q^{d}; q^{d})_\infty }, 
      & \textrm{ if } f > e;
    \end{cases}
    \vspace{3mm} 
  \end{align}
  \begin{align}
  \nonumber 
    \overline{U}_{dk+e, da+f}(1) & = \frac{ ( -q^{d}; q^{d})_\infty 
        ( q^{da+d}, q^{2dk-da-d+2e}, q^{2dk+2e} ; q^{2dk+2e})_\infty }
      { ( q^{2}; q^{2})_\infty ( q^{d}; q^{2d})_\infty }.  
  \end{align}
\end{theorem}

Observe that 
\begin{align}
\nonumber 
  U_{dk+e, da+f}(1) 
  = \sum_{n \geq 0} \left( \sum_{m \geq 0} u_{dk+e, da+f}(m, n) \right) q^n, 
\end{align}
and 
\begin{align}
\nonumber 
  \overline{U}_{dk+e, da+f}(1) 
  = \sum_{n \geq 0} \left( \sum_{m \geq 0} \overline{u}_{dk+e, da+f}(m, n) \right) q^n, 
\end{align}
by definition.  
In other words, $U_{dk+e, da+f}(1)$ and $U_{dk+e, da+f}(1)$ 
are the univariate partition generating functions of the partitions enumerated 
by the first and the second part of the Definition \ref{defMain} 
with any number of parts.  
For $d=e=1$, Theorem \ref{thmOverPtnModd} is Theorem 
\ref{thmRRG-overptn}~\cite{CL-overptns, Lovejoy-RRG-overptn}.  
For $d=e=2$, it is Theorem \ref{thmSangShiYee}~\cite{SSY-parity-overptn}.  
Other cases are new.  

It is straightforward, however tedious, 
to interpret the infinite products on the right hand sides 
of the $q$-series identities in Theorem \ref{thmOverPtnModd} 
as overpartition identities or partition-overpartition identities~\cite{TheBlueBook}.  

\begin{cor}
\label{corollaryOverPtnModdComb}
  For parameters as in \eqref{defParams}, set 
  \begin{align}
  \nonumber 
    \mathcal{U}_{dk+e, da+f}(n) = \sum_{m \geq 0} u_{dk+e, da+f}(m, n) \quad 
    \textrm{ and } \quad 
    \overline{\mathcal{U}}_{dk+e, da+f}(n) 
    = \sum_{m \geq 0} \overline{u}_{dk+e, da+f}(m, n), 
  \end{align}
  i.e. $\mathcal{U}_{dk+e, da+f}(n)$ (respectively, $\overline{\mathcal{U}}_{dk+e, da+f}(n)$) 
  counts the number of overpartitions of $n$ satisfying {\bf (i)} - {\bf (iv)} 
  (respectively, $\mathbf{\overline{(i)}}$ - $\mathbf{\overline{(iv)}}$) 
  in Definition \ref{defMain}, without taking the number of parts into account.  
  
  Let $\mathcal{G}_{dk+d, da+d}(n)$ be the number of partitions of $n$ 
  such that all even parts are multiples of $d$ 
  which are not congruent to 0 or $\pm(da+d) \pmod{2dk+2d}$;  
  
  $\mathcal{H}_{dk+d, da+d}(n)$ be the number of overpartitions of $n$ 
  such that all non-overlined parts are multiples of $d$ 
  which are not congruent to 0 or $\pm(da+d) \pmod{2dk+2d}$; 
  
  $\overline{\mathcal{G}}_{dk+d, da+d}(n)$ be the number of overpartitions of $n$ 
  such that the overlined parts are multiples of $d$, 
  odd non-overlined parts are odd multiples of $d$ 
  which are not congruent to 0 or $\pm(da+d) \pmod{2dk+2d}$; 
  
  $\overline{\mathcal{H}}_{dk+d, da+d}(n)$ be the number of pairs $(\lambda, \mu)$ 
  of a partition $\lambda$ and an overpartition $\mu$ such that 
  the total of sum of parts of $\lambda$ and sum of parts of $\mu$ is $n$, 
  $\lambda$ has even parts only, 
  all parts of $\mu$ are multiples of $d$,
  and the non-overlined parts of $\mu$ are odd multiples of $d$ 
  which are not congruent to 0 or $\pm(da+d) \pmod{2dk+2d}$.  
  
  Then; 
  for $e=d$ and $d$ even, $\mathcal{U}_{dk+d, da+d}(n) = \mathcal{G}_{dk+d, da+d}(n)$; 
  
  for $e = d$ and $d$ odd, $\mathcal{U}_{dk+d, da+d}(n) = \mathcal{H}_{dk+d, da+d}(n)$; 
  
  for $2e = d$,  $\mathcal{U}_{dk+\frac{d}{2}, da+\frac{d}{2}}(n) 
  = \mathcal{G}_{dk+\frac{d}{2}, da+d}(n)$; 
  
  for $e=d$ and $d$ odd, $\overline{\mathcal{U}}_{dk+d, da+f}(n)
  = \overline{\mathcal{G}}_{dk+d, da+d}(n)$;
  
  for $e=d$ and $d$ even, $\overline{\mathcal{U}}_{dk+d, da+f}(n)
  = \overline{\mathcal{H}}_{dk+d, da+d}(n)$;
  
  for $2e=d$, $\overline{\mathcal{U}}_{dk+\frac{d}{2}, da+f}(n)
  = \overline{\mathcal{H}}_{dk+\frac{d}{2}, da+d}(n)$.  
  
  Also; 
  for $e = d$, $d$ even, and $f < e$,
  \begin{align}
  \nonumber 
    & \mathcal{U}_{dk+d, da+f}(n) - \mathcal{U}_{dk+d, da+f}(n-d) \\
  \nonumber 
    & = \mathcal{G}_{dk+d, da+d}(n) - \mathcal{G}_{dk+d, da+d}(n-d-f+e)
    + \mathcal{G}_{dk+d, da}(n-d-f+e) - \mathcal{G}_{dk+d, da}(n-d); 
  \end{align}
  for $e = d$, $d$ even, and $f > e$,
  \begin{align}
  \nonumber 
    & \mathcal{U}_{dk+d, da+f}(n) - \mathcal{U}_{dk+d, da+f}(n-d) \\
  \nonumber 
    & = \mathcal{G}_{dk+d, da+d}(n-f+e) - \mathcal{G}_{dk+d, da+d}(n-d)
    + \mathcal{G}_{dk+d, da+2d}(n) - \mathcal{G}_{dk+d, da+2d}(n-f+e); 
  \end{align}
  for $e = d$, $d$ odd, and $f < e$,
  \begin{align}
  \nonumber 
    & \mathcal{U}_{dk+d, da+f}(n) - \mathcal{U}_{dk+d, da+f}(n-d) \\
  \nonumber 
    & = \mathcal{H}_{dk+d, da+d}(n) - \mathcal{H}_{dk+d, da+d}(n-d-f+e)
    + \mathcal{H}_{dk+d, da}(n-d-f+e) - \mathcal{H}_{dk+d, da}(n-d); 
  \end{align}
  and for $e = d$, $d$ odd, and $f > e$,
  \begin{align}
  \nonumber 
    & \mathcal{U}_{dk+d, da+f}(n) - \mathcal{U}_{dk+d, da+f}(n-d) \\
  \nonumber 
    & = \mathcal{H}_{dk+d, da+d}(n-f+e) - \mathcal{H}_{dk+d, da+d}(n-d)
    + \mathcal{H}_{dk+d, da+2d}(n) - \mathcal{H}_{dk+d, da+2d}(n-f+e).  
  \end{align}
\end{cor}

The paper is organized as follows.  
In section \ref{secProofs}, verification type proofs are given 
for Theorem \ref{thmOverPtnModd}, along with all necessary auxiliary results.  
In section \ref{secConstr}, the construction of series 
as solutions of functional equations in section \ref{secProofs} is outlined.  
The approach is the same as in~\cite{KK-AndrewsStyle}.  
We conclude with some commentary and future work in section \ref{secComm}.  

\section{Proof of Theorem \ref{thmOverPtnModd}}
\label{secProofs} 

\begin{lemma}
\label{lemmaUbarAdjustment}
  With parameters as in \eqref{defParams}, 
  \begin{align}
  \nonumber 
    \overline{U}_{dk+e, da+f}(x) = \overline{U}_{dk+e, da+d}(x).  
  \end{align}
\end{lemma}

\noindent
{\bf Remark:} For $d = 2$, this is Theorem 3.4 in~\cite{SSY-parity-overptn}.  

\begin{proof}
  This is equivalent to 
  \begin{align}
  \nonumber 
    \overline{u}_{dk+e, da+f}(m, n) = \overline{u}_{dk+e, da+d}(m, n).  
  \end{align}
  The only condition in the definition of $\overline{u}_{dk+e, da+f}(m, n)$ 
  in which $f$ appears is $\mathbf{\overline{(i)}}$.  
  Rewriting, we have 
  \begin{align}
  \nonumber 
    (f_1 + f_{\overline{1}}) - d f_{\overline{{1}}} - da \leq f-1.  
  \end{align}
  Because of $\mathbf{\overline{(iii)}}$, 
  the left hand side of the inequality above is a multiple of $d$.  
  On the other hand, $f-1$ may assume values 0, 1, \ldots, $d-1$.  
  For fixed $d$ and $a$, any value of $f-1$ among the listed values 
  will yield the same solutions for $f_1$ and $f_{\overline{1}}$ 
  in that inequality.  
\end{proof}

\begin{lemma}
\label{lemmaUUbarFunclEqns}
  With parameters as in \eqref{defParams}, 
  \begin{align}
  \nonumber 
    & U_{dk+e, da+f}(x) - U_{dk+e, da+f-1}(x) \\
  \nonumber 
    & = \begin{cases}
      (xq)^{da+f-1} \overline{U}_{dk+e, d(k-a)+d}(xq)
      + (xq)^{da+d+f-1} \overline{U}_{dk+e, d(k-a-1)+d}(xq) 
       & \textrm{ if } f \leq e, \\
      (xq)^{da+f-1} \overline{U}_{dk+e, d(k-a-1)+d}(xq)
      + (xq)^{da+d+f-1} \overline{U}_{dk+e, d(k-a-2)+d}(xq) 
       & \textrm{ if } f > e, 
      \end{cases}
  \end{align}
  \begin{align}
  \nonumber 
    \overline{U}_{dk+e, da+d}(x) - \overline{U}_{dk+e, d(a-1)+d}(x) 
    = (xq)^{da} U_{dk+e, d(k-a)+e}(xq)
      + (xq)^{da+d} U_{dk+e, d(k-a-1)+e}(xq).  
  \end{align}
\end{lemma}

\begin{proof}
  We write definitions of $U_{\cdot, \cdot}(x)$ and $\overline{U}_{\cdot, \cdot}(x)$ 
  as double power series, 
  then identify coefficients of $x^m q^n$ for each $m$ and $n$ on either 
  side of the functional equations to see that the 
  lemma is equivalent to the following recurrences.  
  \begin{align}
  \nonumber 
    & u_{dk+e, da+f}(m, n) - u_{dk+e, da+f}(m, n) \\
  \nonumber 
    & = \begin{cases}
      \overline{u}_{dk+e, d(k-a)+d}(m-da-f+1, n-m) & \\ 
      + \overline{u}_{dk+e, d(k-a-1)+d}(m-da-d-f+1, n-m) 
      & \textrm{ if } f \leq e, \\ 
      \overline{u}_{dk+e, d(k-a-1)+d}(m-da-f+1, n-m) & \\ 
      + \overline{u}_{dk+e, d(k-a-2)+d}(m-da-d-f+1, n-m) 
      & \textrm{ if } f > e, 
      \end{cases}
  \end{align}
  \begin{align}
  \nonumber 
    & \overline{u}_{dk+e, da+d}(m, n) - \overline{u}_{dk+e, d(a-1)+d}(m, n) \\
  \nonumber 
    & = u_{dk+e, d(k-a)+e}(m-da, n-m)
      + u_{dk+e, d(k-a-1)+e}(m-da-d, n-m).  
  \end{align}
  Because the proofs are almost the same after the necessary notational changes, 
  we will only prove 
  \begin{align}
  \nonumber 
    & u_{dk+e, da+f}(m, n) - u_{dk+e, da+f}(m, n) \\
  \nonumber 
    & = \overline{u}_{dk+e, d(k-a)+d}(m-da-f+1, n-m)
      + \overline{u}_{dk+e, d(k-a-1)+d}(m-da-d-f+1, n-m) 
  \end{align}
  for $f \leq e$.  
  
  All partitions enumerated by $u_{dk+e, da+f}(m, n)$ and $u_{dk+e, da+f-1}(m, n)$
  satisfy {\bf ({ii})}, {\bf ({iii})}, and {\bf ({iv})} in Definition \ref{defMain}.  
  Since 
  \begin{align}
  \nonumber 
    f_1 \leq da + (f-1) - 1 + (d-1)f_{\overline{1}} 
    \Rightarrow f_1 \leq da + f - 1 + (d-1)f_{\overline{1}}, 
  \end{align}
  all partitions counted by $u_{dk+e, da+f-1}(m, n)$ 
  are also counted by $u_{dk+e, da+f}(m, n)$.  
  The extra partitions in $u_{dk+e, da+f}(m, n)$ then satisfy 
  {\bf ({ii})}, {\bf ({iii})}, and {\bf ({iv})} in Definition \ref{defMain}
  as well as 
  \begin{align}
  \nonumber 
    f_1 = da + f - 1 + (d-1)f_{\overline{1}}, \quad \textrm{ or rewriting, } \quad 
    f_1 + f_{\overline{1}} = da + f - 1 + df_{\overline{1}}. 
  \end{align}
  Since $f_{\overline{1}} = $ 0 or 1, we have two cases.  
  \begin{align}
  \nonumber 
    & f_{\overline{1}} = 0 \Rightarrow f_1 + f_{\overline{1}} = da+f-1,  \\
  \nonumber 
    & f_{\overline{1}} = 1 \Rightarrow f_1 + f_{\overline{1}} = da+d+f-1.  
  \end{align}
  In either case, we'll delete all 1's and the $\overline{1}$, if any, 
  and subtract one from the remaining parts 
  in each of the aforementioned extra partitions.  
  After the deletion of the 1's and the $\overline{1}$, if any, 
  the remaining parts are at least two, 
  so subtracting one from all of them 
  will not introduce anomalies such as zeroes or negative parts.  
  The transformed partitions will satisfy $\mathbf{\overline{(ii)}}$ and $\mathbf{\overline{(iii)}}$ 
  because subtracting one from all parts will switch parities.  
  $\mathbf{\overline{(iv)}}$ is also satisfied, because it is the same as {\bf ({iv})}.  
  
  In the former case ($f_{\overline{1}} = 0 \Rightarrow f_1 + f_{\overline{1}} = da+f-1$), 
  {\bf ({iv})} for $l = 1$ before the deletion of 1's and the subtraction of ones is 
  \begin{align}
  \nonumber 
    f_1 + f_{\overline{1}} + f_2 \leq dk + e - 1 + (d-1)f_{\overline{2}} 
    \rightarrow 
    f_2 \leq d(k-a) + (e-f+1) - 1 + (d-1)f_{\overline{2}}.  
  \end{align}
  After the deletion of 1's and subtraction of ones, this becomes 
  \begin{align}
  \nonumber 
    f_1 \leq d(k-a) + (e-f+1) - 1 + (d-1)f_{\overline{1}}.  
  \end{align}
  This is an instance of $\mathbf{\overline{(i)}}$ in Definition \ref{defMain}.  
  Thus, these partitions after the described transformation 
  are enumerated by a $\overline{u}_{dk+e, d(k-a)+(e-f+1)}(\cdot, \cdot)$.  
  Since $1 \leq f \leq e \leq d$, we have $1 \leq e-f+1 \leq d$, 
  the last displayed inequality is equivalent to 
  \begin{align}
  \nonumber 
    f_1 \leq d(k-a) + d - 1 + (d-1)f_{\overline{1}}, 
  \end{align}
  and thus the said partitions are also counted by 
  $\overline{u}_{dk+e, d(k-a)+d}(\cdot, \cdot)$.  
  
  If we regard the deletion of the 1's as subtracting ones from 1's, 
  hence making them zeroes, as well; 
  we see that the weight of these excess partitions are reduced by the number of parts, $m$.  
  Thus, the new weight is $n-m$ for each transformed partition.  
  At the same time, we discard the $da+f-1$ 1's we deleted, 
  so the new number of parts is $m - da - f+1$.  
  In conclusion, in the former case we have considered, 
  each of the transformed partitions is counted by 
  $\overline{u}_{dk+e, d(k-a)+d}(m-da-f+1, n-m)$.  
  
  If we trace our steps back, 
  we see that the process yields a one-to-one correspondence.  
  In other words, beginning with any partition counted by 
  $\overline{u}_{dk+e, d(k-a)+d}(m-da-f+1, n-m)$, 
  we add one to all of the parts and append exactly $da+f-1$ non-overlined 1's, 
  and we end up with a partition enumerated by $u_{dk+e, da+f}(m, n)$, 
  but not $u_{dk+e, da+f-1}(m, n)$.  
  
  In the latter case 
  ($f_{\overline{1}} = 1 \Rightarrow f_1 + f_{\overline{1}} = da+d+f-1$), 
  the only difference is the extra $d$ on the right hand side.  
  Therefore, the toll $f_{\overline{1}} = 1$ takes is that the final partition 
  being counted by $\overline{u}_{dk+e, d(k-a-1)+d}(m-da-d-f+1, n-m)$, 
  when we adjust the procedure and the inequalities accordingly.  
  
  We have demonstrated one of the functional equations. 
  As stated at the beginning of the proof, 
  the others are very similar.  
\end{proof}

To solve functional equations 
using the defining $q$-difference equations principle~\cite{Andrews-parity}, 
we need initial conditions.  
At this point, we relax the condition on $f$ a little, and allow $f = 0$.  
It is convenient to keep in mind that $f$ only stands for a residue class, 
so $f = 0$ will correspond to $f = d$.  
The asserted solutions further down will explicate this point of view.  

Also, please notice that the definitions of $U_{dk+e, da+f}(x)$ 
and $\overline{U}_{dk+e, da+f}(x)$ imply that 
\begin{align}
\nonumber 
  u_{dk+e, da+f}(m,n) = \overline{u}_{dk+e, da+f}(m,n) = 0 
  \textrm{ for } m < 0 \textrm{ or } n < 0.  
\end{align}
This is easy to back up combinatorially.  
As $n$ accounts for the number being partitioned, any partition number 
for negative integers is zero.  
And, since $m$ keeps track of the number of parts in a partition,  
i.e. it \emph{counts} something, it has to be a non-negative integer.  

\begin{lemma}
\label{lemmaUUbarInitVals}
  With parameters as in \eqref{defParams}, 
  \begin{align}
  \nonumber 
    & U_{dk+e, da+f}(0) = \overline{U}_{dk+e, da+f}(0) = 1, \\
  \nonumber 
    & U_{dk+e, 0}(x) = \overline{U}_{dk+e, 0}(x) = 0.  
  \end{align}
\end{lemma}

\begin{proof}
  When we plug in $x = 0$, 
  the generating functions in Definition \ref{defMain} become 
  \begin{align}
  \nonumber 
    & U_{dk+e, da+f}(0) = \sum_{n \geq 0} u_{dk+e, da+f}(n, 0) q^n, \\
  \nonumber 
    & \textrm{and }
    \overline{U}_{dk+e, da+f}(0) = \sum_{n \geq 0} \overline{u}_{dk+e, da+f}(n, 0) q^n.  
  \end{align}
  Now, 
  \begin{align}
  \nonumber 
    u_{dk+e, da+f}(n, 0)
    = \overline{u}_{dk+e, da+f}(n, 0) 
    = \begin{cases}
      1 & \textrm{ if } n = 0, \\ 
      0 & \textrm{ if } n > 0, 
      \end{cases}
  \end{align}
  since the only partition having zero number of parts is the empty partition of zero.  
  Please observe that conditions {\bf {(i)}}-{\bf {(iv)}}, and 
  $\mathbf{\overline{(i)}}$-$\mathbf{\overline{(iv)}}$ in Definition \ref{defMain} hold when 
  $f_l = f_{\overline{l}} = 0$ for all $l \geq 0$.  
  This proves the first line of identities.  
  
  When $da+f = 0$, {\bf ({i})} in Definition \ref{defMain} becomes 
  $f_1 \leq -1 + (d-1)f_{\overline{1}}$.  
  For $f_{\overline{1}} = 0$, this reduces to $f_1 \leq -1$, 
  which is impossible.  
  For $f_{\overline{1}} = 1$, $f_1 \leq d-2$ contradicts 
  {\bf ({ii})} in Definition \ref{defMain}, 
  i.e. $f_1 \leq d-1$.  
  In either case, there cannot be any such partions, 
  so that $u_{dk+e, 0}(m, n) = 0$ for any $m$ and $n$.  
  
  $da+f = 0$ $\mathbf{\overline{(i)}}$ in Definition \ref{defMain} yields 
  $f_1 \leq -1 + (d-1)f_{\overline{1}}$ again.  
  The $f_{\overline{1}} = 0$ possibility reduces the inequality to $f_1 \leq -1$ 
  as in the previous paragraph, and this is similarly impossible.  
  When $f_{\overline{1}} = 1$, $f_1 + f_{\overline{1}} \leq d-1$.  
  This case, together with $\mathbf{\overline{(iii)}}$ in Definition \ref{defMain} 
  requires that $f_1$ $=-1$, $-d-1$, $-2d-1$, \ldots, 
  since $f_1 + f_{\overline{1}} \equiv \pmod{d}$ and $f_{\overline{1}} = 1$.  
  In either case, $f_1$ is forced to be negative; 
  therefore, $\overline{u}_{dk+e, 0}(m, n) = 0$ for any $m$ and $n$.  
  
  The last two paragraphs establish the second line of identities, 
  and conclude the proof.  
\end{proof}

Next, we claim and verify a set of solutions 
to the functional equations and initial conditions 
given by Lemmas \ref{lemmaUbarAdjustment}-\ref{lemmaUUbarInitVals}.  
These in turn will prove Theorem \ref{thmOverPtnModd}, that is the main result.  

Let's set forth the following $q$-hypergeometric terms, 
using the parameters described in \eqref{defParams}, except for $a$. 
\begin{align}
\nonumber 
  \alpha^f_n(x) & = (-1)^n x^{(dk+e)n} (xq^{n+1})^{f-e} q^{2(dk+e) \binom{n+1}{2}} 
   \frac{ ( (xq)^d; q^{2d})_\infty ( -q^d; q^d)_n ( -(xq^{n+1})^d; q^d)_\infty }
    { ( xq; q^2)_\infty ( q^d; q^d)_n ( (xq^{n+1})^d; q^d)_\infty  } \\ 
\label{defAlphaf}
  & \times \begin{cases}
    \left[ \frac{ ( (xq)^{e-f} - (xq)^d ) }{ ( 1 - (xq)^d ) } 
      + q^{nd} \frac{ ( (xq)^d - (xq)^{d+e-f} ) }{ ( 1 - (xq)^d ) } \right] , 
      & \textrm{ if } f < e, \\ 
    1 , & \textrm{ if } f = e, \\ 
    \left[ \frac{ ( 1 - (xq)^{d+e-f} ) }{ ( 1 - (xq)^d ) } 
      + q^{-nd} \frac{ ( (xq)^{e-f} - 1 ) }{ ( 1 - (xq)^d ) } \right] , 
      & \textrm{ if } f > e.  
  \end{cases}
\end{align}
\begin{align}
\nonumber 
  \beta^f_n(x) & = - (-1)^n x^{(dk+e)n} q^{(e-f)n} q^{2(dk+e) \binom{n+1}{2}} 
   \frac{ ( (xq)^d; q^{2d})_\infty ( -q^d; q^d)_n ( -(xq^{n+1})^d; q^d)_\infty }
    { ( xq; q^2)_\infty ( q^d; q^d)_n ( (xq^{n+1})^d; q^d)_\infty  } \\ 
\label{defBetaf}
  & \times \begin{cases}
    \left[ \frac{ ( (xq)^{e-f} - (xq)^d ) }{ ( 1 - (xq)^d ) } 
      + q^{-nd} \frac{ ( 1 - (xq)^{e-f} ) }{ ( 1 - (xq)^d ) } \right] , 
      & \textrm{ if } f < e, \\ 
    1 , & \textrm{ if } f = e, \\ 
    \left[ \frac{ ( 1 - (xq)^{d+e-f} ) }{ ( 1 - (xq)^d ) } 
      + q^{nd} \frac{ ( (xq)^{d+e-f} - (xq)^d ) }{ ( 1 - (xq)^d ) } \right] , 
      & \textrm{ if } f > e.  
  \end{cases}
\end{align}
\begin{align}
\label{defAlphaBetaBar}
  \overline{\alpha}_n(x) = - \overline{\beta}_n(x) 
  = (-1)^n x^{(dk+e)n} q^{ 2(dk+e)\binom{n+1}{2} }
   \frac{ ( (xq^2)^d; q^{2d})_\infty ( -q^d; q^d)_n ( -(xq^{n+1})^d; q^d)_\infty }
    { ( xq^2; q^2)_\infty ( q^d; q^d)_n ( (xq^{n+1})^d; q^d)_\infty  }.  
\end{align}
\noindent
{\bf Remark: } It is easy to verify that if we use either of the $f \neq e$ cases 
for $\alpha^e_n(x)$ and $\beta^e_n(x)$, we arrive at the $f = e$ case.  
Also, using the $f<e$ case for $\alpha^0_n(x)$ and $\beta^0_n(x)$, respectively,
and the $f > e$ case for $\alpha^d_n(x)$ and $\beta^d_n(x)$, respectively,
yield the same $q$-hypergeometric term.  
As indicated above, the role of $f$ is that of a residue class, 
and $\alpha^f_n(x)$ and $\beta^f_n(x)$ behave nicely with the extremes.  

\begin{lemma}
\label{lemmaUUbarClaimedSolns}
  With parameters as described in \eqref{defParams}, for $e = d$ or $2e = d$, 
  and the $q$-hypergeometric terms as defined in \eqref{defAlphaf}-\eqref{defAlphaBetaBar}, 
  \begin{align}
  \nonumber 
    U_{dk+e, da+f}(x) 
    & = \sum_{n \geq 0} \alpha^f_n(x) q^{-n(da+f)} 
      + \beta^f_n(x) (xq^{n+1})^{da+f}, \\ 
  \nonumber 
    \overline{U}_{dk+e, da+f}(x) 
    & = \sum_{n \geq 0} \overline{\alpha}_n(x) q^{-n(da+d)} 
      + \overline{\beta}_n(x) (xq^{n+1})^{da+d}.  
  \end{align}
\end{lemma}

\begin{proof}
  Thanks to the defining $q$-difference equations principle~\cite{Andrews-parity}
  Lemmas \ref{lemmaUbarAdjustment}-\ref{lemmaUUbarInitVals} 
  uniquely determine the power series for $U_{dk+e, da+f}(x)$ 
  and $\overline{U}_{dk+e, da+f}(x)$.  
  Thus, all we need to do is to verify that the claimed solutions in the statement 
  satisfy the functional equations in Lemma \ref{lemmaUUbarFunclEqns} 
  and the initial conditions in Lemma \ref{lemmaUUbarInitVals}.  
  
  The identities in Lemma \ref{lemmaUbarAdjustment} are satisfied by definition.  
  
  The computational verification of each of the functional equations in 
  Lemma \ref{lemmaUUbarFunclEqns} is similar to each other.  
  We sketch one below.  
  Following the remark after the definitions \eqref{defAlphaf}-\eqref{defAlphaBetaBar}, 
  for $1 \leq f \leq e$, 
  \begin{align}
  \nonumber 
    & U_{dk+e, da+f}(x) - U_{dk+e, da+f-1}(x) \\
  \nonumber 
      & (xq)^{da+f-1} \overline{U}_{dk+e, d(k-a)+d}(xq)
      + (xq)^{da+d+f-1} \overline{U}_{dk+e, d(k-a-1)+d}(xq), 
  \end{align}
  we replace the functions with the as yet alleged solutions.  
  \begin{align}
  \nonumber 
    & \sum_{n \geq 0} \alpha^f_n(x) q^{-n(da+f)} 
      + \beta^f_n(x) (xq^{n+1})^{da+f} \\
  \nonumber 
    & - \sum_{n \geq 0} \alpha^{f-1}_n(x) q^{-n(da+f-1)} 
      + \beta^{f-1}_n(x) (xq^{n+1})^{da+f-1} \\
  \nonumber 
      & \stackrel{?}{=} (xq)^{da+f-1} 
      \sum_{n \geq 0} \overline{\alpha}_n(xq) q^{-n(d(k-a)+d)} 
      + \overline{\beta}_n(xq) (xq^{n+2})^{d(k-a)+d} \\
  \nonumber 
      & + (xq)^{da+d+f-1} 
      \sum_{n \geq 0} \overline{\alpha}_n(xq) q^{-n(d(k-a-1)+d)} 
      + \overline{\beta}_n(xq) (xq^{n+2})^{d(k-a-1)+d}.  
  \end{align}
  Because of the $(q^d; q^d)_n$ in the denominator, 
  $\alpha^f_n(x)$ $= \beta^f_n(x)$ $=\overline{\alpha}_n(x)$ 
  $=\overline{\beta}_n(x)$ $=0$ for $n < 0$~\cite{GR}.  
  So we can shift the index for the second type of terms 
  inside the sums on the right hand side 
  without changing the domain for the summation index $n$.  
  We also regroup the terms on both sides.  
  \begin{align}
  \nonumber 
    & \sum_{n \geq 0} 
      \left( \alpha^f_n(x) q^{-n(da+f)} 
        - \alpha^{f-1}_n(x) q^{-n(da+f-1)} \right) \\ 
  \nonumber 
    & + \sum_{n \geq 0} 
      \left( \beta^f_n(x) (xq^{n+1})^{da+f} 
        - \beta^{f-1}_n(x) (xq^{n+1})^{da+f-1} \right) \\ 
  \nonumber 
      & \stackrel{?}{=} \sum_{n \geq 0} 
      \left( (xq)^{da+f-1} q^{-n(d(k-a)+d)} 
        + (xq)^{da+d+f-1} q^{-n(d(k-a-1)+d)} \right) \overline{\alpha}_n(xq) \\
  \nonumber 
      & + \sum_{n \geq 0} 
      \left( (xq)^{da+f-1} (xq^{n+1})^{d(k-a)+d} 
        + (xq)^{da+d+f-1} (xq^{n+1})^{d(k-a-1)+d} \right) \overline{\beta}_{n-1}(xq).  
  \end{align}
  Then, we verify the sufficient (but clearly not necessary) 
  \begin{align}
  \nonumber 
    & \left( \alpha^f_n(x) q^{-n(da+f)} 
        - \alpha^{f-1}_n(x) q^{-n(da+f-1)} \right) \\ 
  \nonumber 
      & = \left( (xq)^{da+f-1} (xq^{n+1})^{d(k-a)+d} 
        + (xq)^{da+d+f-1} (xq^{n+1})^{d(k-a-1)+d} \right) \overline{\beta}_{n-1}(xq), 
  \end{align}
  and 
  \begin{align}
  \nonumber 
    & \left( \beta^f_n(x) (xq^{n+1})^{da+f} 
        - \beta^{f-1}_n(x) (xq^{n+1})^{da+f-1} \right) \\ 
  \nonumber 
      & = \left( (xq)^{da+f-1} q^{-n(d(k-a)+d)} 
        + (xq)^{da+d+f-1} q^{-n(d(k-a-1)+d)} \right) \overline{\alpha}_n(xq).  
  \end{align}
  for each $n \geq 0$.  
  All are straightforward elementary calculations.  
  These show that the claimed solutions satisfy 
  the functional equations in Lemma \ref{lemmaUUbarFunclEqns}.  
  
  The first line of identities in Lemma \ref{lemmaUUbarInitVals}
  are also easily seen to hold.  
  For the second line of identities, 
  we verify that 
  \begin{align}
  \nonumber 
    \alpha^0_n(x) + \beta^0_n(x) 
    = \overline{\alpha}_n(x) + \overline{\beta}_n(x) 
    = 0.  
  \end{align}
  The only place where we need $e = d$ or $2e = d$ is the vanishing of the first sum.  
  This wraps up the proof.  
\end{proof}

\begin{proof}[proof of Theorem \ref{thmOverPtnModd}]
  We present the proof of only one case, the others being completely analogous.  
  Plugging in $x = 1$ in Lemma \ref{lemmaUUbarClaimedSolns} for $f < e$, 
  combining some finite and infinite products, 
  and factoring out whichever does not depend on the summation index, 
  we have 
  \begin{align}
  \nonumber 
    & U_{dk+e, da+f}(1) = \frac{ (q^{d}; q^{2d})_\infty ( -q^{d}; q^{d})_\infty }
      { ( q; q^2)_\infty ( q^d; q^d )_\infty } \\
  \nonumber 
    & \times \left( \sum_{n \geq 0} (-1)^n  q^{ 2(dk+e) \binom{n+1}{2} } 
      q^{ -n(da+e) } \left\{ \frac{ (1 - q^{d+f-e}) }{ (1 - q^d) }   
    - q^{ -n(-d) } \frac{ (q^{d+f-e} - q^d) }{ (1 - q^d) } \right\} \right. \\
  \nonumber 
    & \left. - (-1)^n  q^{ 2(dk+e) \binom{n+1}{2} } 
      q^{ (n+1)(da+e) } \left\{ \frac{ (1 - q^{d+f-e}) }{ (1 - q^d) }   
    - q^{ (n+1)(-d) } \frac{ (q^{d+f-e} - q^d) }{ (1 - q^d) } \right\} \right)
  \end{align}
  \begin{align}
  \nonumber 
    & = \frac{ 1 }
      { ( q; q^2)_\infty ( q^d; q^d )_\infty } 
    \left\{ \frac{ (1 - q^{d+f-e}) }{ (1 - q^d) } 
      \left( \sum_{n = -\infty}^\infty (-1)^n q^{ 2(dk+e)\binom{n+1}{2} } 
        q^{n(da+e)} \right) \right. \\ 
  \nonumber 
    & \left. + \frac{ (q^{d+f-e} - q^d) }{ (1 - q^d) } 
      \left( \sum_{n = -\infty}^\infty (-1)^n q^{ 2(dk+e)\binom{n+1}{2} } 
        q^{n(da-d+e)} \right) \right\}.  
  \end{align}
  $(q^{d}; q^{2d})_\infty ( -q^{d}; q^{d})_\infty = 1$ 
  is Euler's partition identity~\cite{TheBlueBook} after $q \to q^d$.  
  Applying Jacobi's triple product identity~\cite{GR} on each of the bilateral sums 
  inside the last pair of set braces finishes the proof.  
\end{proof}

\section{The Construction of Solutions to Functional Equations at a Glance}
\label{secConstr} 

The idea is the same as in~\cite{KK-AndrewsStyle}.  
The motivation and the notation comes from Andrews' 
papers~\cite{Andrews-RRG-analytic, Andrews-posets-RRG}.  
With parameters as in \eqref{defParams}, 
we first assume that the generating functions have the following form.  
\begin{align}
\nonumber 
  U_{dk+e, da+f}(x) 
  = \sum_{ n \geq 0 } \alpha_n(x) x^{(da+f)B} q^{(da+f)C} q^{n(da+f)D}
  + \beta_n(x) x^{(da+f)E} q^{(da+f)F} q^{n(da+f)G}, 
\end{align}
and
\begin{align}
\nonumber 
  \overline{U}_{dk+e, da+d}(x) 
  = \sum_{ n \geq 0 } \overline{\alpha}_n(x) 
    x^{(da+d)\overline{B}} q^{(da+d)\overline{C}} q^{n(da+d)\overline{D}}
  + \overline{\beta}_n(x) 
    x^{(da+d)\overline{E}} q^{(da+d)\overline{F}} q^{n(da+d)\overline{G}}, 
\end{align}
for $q$-hypergeometric terms $\alpha_n(x)$, $\beta_n(x)$, 
$\overline{\alpha}_n(x)$, $\overline{\beta}_n(x)$, 
and integers $B$, $C$, $D$, $E$, $F$, $G$, 
$\overline{B}$, $\overline{C}$, $\overline{D}$, 
$\overline{E}$, $\overline{F}$, and $\overline{G}$ to be determined.  
It is important to have the $q$-hypergeometric terms 
($\alpha_n(x)$, $\beta_n(x)$, $\overline{\alpha}_n(x)$, $\overline{\beta}_n(x)$) 
to be independent of the second index $(da+f)$.  

To construct the $q$-hypergeometric terms 
$\alpha_n(x)$, $\beta_n(x)$, $\overline{\alpha}_n(x)$, $\overline{\beta}_n(x)$, 
and to determine at least some of the unknown integer exponents, 
we use the functional equations in Lemma \ref{lemmaUUbarFunclEqns}.  
Here, we imitate the framework of proofs in~\cite{Andrews-RRG-analytic}.  
As an example, we write the functional equation 
\begin{align}
\nonumber 
  \overline{U}_{dk+e, da+d}(x) - \overline{U}_{dk+e, d(a-1)+d}(x) 
  = (xq)^{da} U_{dk+e, d(k-a)+e}(xq)
    + (xq)^{da+d} U_{dk+e, d(k-a-1)+e}(xq)
\end{align}
as 
\begin{align}
\nonumber 
  & \sum_{ n \geq 0 } \overline{\alpha}_n(x) 
    x^{(da+d)\overline{B}} q^{(da+d)\overline{C}} q^{n(da+d)\overline{D}}
  + \overline{\beta}_n(x) 
    x^{(da+d)\overline{E}} q^{(da+d)\overline{F}} q^{n(da+d)\overline{G}} \\ 
\nonumber 
  & - \left( \sum_{ n \geq 0 } \overline{\alpha}_n(x) 
    x^{(d(a-1)+d)\overline{B}} q^{(d(a-1)+d)\overline{C}} q^{n(d(a-1)+d)\overline{D}}
  + \overline{\beta}_n(x) 
    x^{(d(a-1)+d)\overline{E}} q^{(d(a-1)+d)\overline{F}} q^{n(d(a-1)+d)\overline{G}} \right)
\end{align}
\begin{align}
\nonumber 
  & = (xq)^{da} \Bigg( \sum_{ n \geq 0 } \alpha_n(xq) 
    (xq)^{(d(k-a)+e)B} q^{(d(k-a)+e)C} q^{n(d(k-a)+e)D} \\ 
\nonumber 
  & \qquad + \beta_n(xq) 
    (xq)^{(d(k-a)+e)E} q^{(d(k-a)+e)F} q^{n(d(k-a)+e)G} \Bigg) \\
\nonumber 
  & + (xq)^{da+d} \Bigg( \sum_{ n \geq 0 } \alpha_n(xq) 
    (xq)^{(d(k-a-1)+e)B} q^{(d(k-a-1)+e)C} q^{n(d(k-a-1)+e)D} \\
\nonumber 
   & \qquad + \beta_n(xq) 
    (xq)^{(d(k-a-1)+e)E} q^{(d(k-a-1)+e)F} q^{n(d(k-a-1)+e)G} \Bigg).  
\end{align}
Then, we rearrange the terms, and shift the index in one of the sums on the left as 
\begin{align}
\nonumber 
  & \sum_{n \geq 0} \Bigg( \overline{\alpha}_n(x) 
    x^{(da+d)\overline{B}} q^{(da+d)\overline{C}} q^{n(da+d)\overline{D}} 
  - \overline{\alpha}_n(x) x^{(d(a-1)+d)\overline{B}} q^{(d(a-1)+d)\overline{C}} 
    q^{n(d(a-1)+d)\overline{D}} \Bigg) \\ 
\nonumber 
  & + \sum_{n \geq 0} \Bigg( \overline{\beta}_n(x) 
    x^{(da+d)\overline{E}} q^{(da+d)\overline{F}} q^{n(da+d)\overline{G}} 
  - \overline{\beta}_n(x) x^{(d(a-1)+d)\overline{E}} q^{(d(a-1)+d)\overline{F}} 
    q^{n(d(a-1)+d)\overline{G}} \Bigg)
\end{align}
\begin{align}
\nonumber 
  & = \sum_{n \geq 0} \Bigg( (xq)^{da} \alpha_n(xq) 
    (xq)^{(d(k-a)+e)B} q^{(d(k-a)+e)C} q^{n(d(k-a)+e)D} \\ 
\nonumber 
  & \qquad + (xq)^{da+d} \alpha_n(xq) 
    (xq)^{(d(k-a-1)+e)B} q^{(d(k-a-1)+e)C} q^{n(d(k-a-1)+e)D} \Bigg) \\ 
\nonumber 
  & + \sum_{n \geq 1} \Bigg( (xq)^{da} \beta_{n-1}(xq) 
    (xq)^{(d(k-a)+e)B} q^{(d(k-a)+e)C} q^{(n-1)(d(k-a)+e)D} \\ 
\nonumber 
  & \qquad + (xq)^{da+d} \beta_{n-1}(xq) 
    (xq)^{(d(k-a-1)+e)B} q^{(d(k-a-1)+e)C} q^{(n-1)(d(k-a-1)+e)D} \Bigg).  
\end{align}
Finally, we identify 
\begin{align}
\nonumber 
  & \overline{\alpha}_n(x) 
    x^{(da+d)\overline{B}} q^{(da+d)\overline{C}} q^{n(da+d)\overline{D}} 
  - \overline{\alpha}_n(x) x^{(d(a-1)+d)\overline{B}} q^{(d(a-1)+d)\overline{C}} 
    q^{n(d(a-1)+d)\overline{D}} \\
\nonumber 
  & = (xq)^{da} \beta_{n-1}(xq) 
    (xq)^{(d(k-a)+e)B} q^{(d(k-a)+e)C} q^{(n-1)(d(k-a)+e)D} \\ 
\nonumber 
  & \qquad + (xq)^{da+d} \beta_{n-1}(xq) 
    (xq)^{(d(k-a-1)+e)B} q^{(d(k-a-1)+e)C} q^{(n-1)(d(k-a-1)+e)D}
\end{align}
for each $n \geq 1$, and 
\begin{align}
\nonumber 
  & \overline{\beta}_n(x) 
    x^{(da+d)\overline{E}} q^{(da+d)\overline{F}} q^{n(da+d)\overline{G}} 
  - \overline{\beta}_n(x) x^{(d(a-1)+d)\overline{E}} q^{(d(a-1)+d)\overline{F}} 
    q^{n(d(a-1)+d)\overline{G}} \\ 
\nonumber 
  & = (xq)^{da} \alpha_n(xq) 
    (xq)^{(d(k-a)+e)B} q^{(d(k-a)+e)C} q^{n(d(k-a)+e)D} \\ 
\nonumber 
  & \qquad + (xq)^{da+d} \alpha_n(xq) 
    (xq)^{(d(k-a-1)+e)B} q^{(d(k-a-1)+e)C} q^{n(d(k-a-1)+e)D} 
\end{align}
for each $n \geq 0$.  
Unfortunately, this yields inconsistent equations.  

So, as in~\cite{KK-AndrewsStyle}, 
we relax the independence on $f$ of $\alpha_n(x)$ and of $\beta_n(x)$.  
We still insist on independence on $a$ of those $q$-hypergeometric terms.  
Thus, the final pair of identifications rather look like 
\begin{align}
\nonumber 
  & \overline{\alpha}_n(x) 
    x^{(da+d)\overline{B}} q^{(da+d)\overline{C}} q^{n(da+d)\overline{D}} 
  - \overline{\alpha}_n(x) x^{(d(a-1)+d)\overline{B}} q^{(d(a-1)+d)\overline{C}} 
    q^{n(d(a-1)+d)\overline{D}} \\
\nonumber 
  & = (xq)^{da} \beta^e_{n-1}(xq) 
    (xq)^{(d(k-a)+e)E_e} q^{(d(k-a)+e)F_e} q^{(n-1)(d(k-a)+e)G_e} \\ 
\nonumber 
  & \qquad + (xq)^{da+d} \beta^e_{n-1}(xq) 
    (xq)^{(d(k-a-1)+e)E_e} q^{(d(k-a-1)+e)F_e} q^{(n-1)(d(k-a-1)+e)G_e}
\end{align}
and 
\begin{align}
\nonumber 
  & \overline{\beta}_n(x) 
    x^{(da+d)\overline{E}} q^{(da+d)\overline{F}} q^{n(da+d)\overline{G}} 
  - \overline{\beta}_n(x) x^{(d(a-1)+d)\overline{E}} q^{(d(a-1)+d)\overline{F}} 
    q^{n(d(a-1)+d)\overline{G}} \\ 
\nonumber 
  & = (xq)^{da} \alpha^e_n(xq) 
    (xq)^{(d(k-a)+e)B_e} q^{(d(k-a)+e)C_e} q^{n(d(k-a)+e)D_e} \\ 
\nonumber 
  & \qquad + (xq)^{da+d} \alpha^e_n(xq) 
    (xq)^{(d(k-a-1)+e)B_e} q^{(d(k-a-1)+e)C_e} q^{n(d(k-a-1)+e)D_e} 
\end{align}
for the stated $n$'s.  
Some experimentation shows that each of the $q$-hypergeometric terms 
$\alpha^f_n(x)$, $\beta^f_n(x)$, $\overline{\alpha}_n(x)$, and $\overline{\beta}_n(x)$
are of the form 
\begin{align}
\nonumber 
  & x^{\textrm{linear exponent in } n} \cdot q^{\textrm{quadratic exponent in } n} 
  \cdot \textrm{a ratio of various } q \textrm{-Pochhammer symbols} \\
\nonumber 
  & \times  
  \cdot \textrm{ a rational function in  } x \textrm{ and } q
  \cdot \left( \alpha^e_0(xq^{2n(+1)}) 
    \textrm{ or } \overline{\alpha}_0(xq^{2n(+1)}) \right).  
\end{align}
We will further assume that $\alpha^e_0(x)$ and $\overline{\alpha}_0(x)$ 
are analytic with constant term 1.  
The rational function has a double power series in $x$ and $q$ with constant term 1.  
Later in the computations, we will be plugging in $x = 0$, 
therefore $B_f$, $\overline{B}$, $E_f$ and $\overline{E}$ should be non-negative.  
Furthermore, we will want $U_{\cdot, \cdot}(0)$ $=\overline{U}_{\cdot, \cdot}(0)$ $= 1$.  
This implies that either $(B_f, C_f)$ or $(E_f, F_f)$ must be $(0, 0)$.  
Same is true for the pairs $(\overline{B}, \overline{C})$ 
and $(\overline{E}, \overline{F})$.  
We set $(B_f, C_f)$ $=(\overline{B}, \overline{C})$ $=(0,0)$.  
In fact, we can show that this is possible without any loss of generality.  
Our sample identifications now read: 
\begin{align}
\nonumber 
  & \overline{\alpha}_n(x) q^{n(da+d)\overline{D}} 
  - \overline{\alpha}_n(x) q^{n(d(a-1)+d)\overline{D}} \\
\nonumber 
  & = (xq)^{da} \beta^e_{n-1} (xq)^{(d(k-a)+e)E_e} q^{(d(k-a)+e)F_e} 
  q^{(n-1)(d(k-a)+e)G_e} \\ 
\nonumber 
  & + (xq)^{da+d} \beta^e_{n-1} (xq)^{(d(k-a-1)+e)E_e} q^{(d(k-a-1)+e)F_e} 
  q^{(n-1)(d(k-a-1)+e)G_e}, 
\end{align}
and 
\begin{align}
\nonumber 
  & \overline{\beta}_n(x) 
    x^{(da+d)\overline{E}} q^{(da+d)\overline{F}} q^{n(da+d)\overline{G}} 
  - \overline{\beta}_n(x) x^{(d(a-1)+d)\overline{E}} q^{(d(a-1)+d)\overline{F}} 
    q^{n(d(a-1)+d)\overline{G}} \\ 
\nonumber 
  & = (xq)^{da} \alpha^e_n(xq) q^{n(d(k-a)+e)D_e} 
  + (xq)^{da+d} \alpha^e_n(xq) q^{n(d(k-a-1)+e)D_e}.  
\end{align}
If we enforce independence on $a$ of the $q$-hypergeometric terms at this point, 
we will deduce that
\begin{align}
\nonumber 
  -D_f = -\overline{D} 
  = E_f = \overline{E} 
  = F_f = \overline{F} 
  = G_f = \overline{G} 
  = 1, 
\end{align}
which makes the sample identifications
\begin{align}
\nonumber 
  \overline{\alpha}_n(x) q^{nd} - \overline{\alpha}_n(x)
  = \beta^e_{n-1}(xq)q^{(n-1)(dk+e)}
  + (xq)^{d} \beta^e_{n-1}(xq)q^{(n-1)(d(k-1)+e)}, 
\end{align}
and 
\begin{align}
\nonumber 
\nonumber 
  & \overline{\beta}_n(x) x^{d} q^{d} q^{nd} - \overline{\beta}_n(x) 
  = \alpha^e_n(xq) q^{n(dk+e)} 
  + (xq)^{d} \alpha^e_n(xq) q^{n(d(k-1)+e)}.  
\end{align}
This speeds up the computations considerably.  
Bringing in other identifications, as well, we have
\begin{align}
\nonumber 
  & \begin{bmatrix}
    1 & & & -q^n \\
    -q^n & 1 & & \\
    & & \ddots & \\
    & & -q^n & 1
  \end{bmatrix}
  \begin{bmatrix}
    \alpha^1_n(x) \\ \vdots \\ \alpha^d_n(x) 
  \end{bmatrix} \\
\nonumber 
  & = (xq^{n+1})^{dk} (xq)^{-1} q^{-nd} (1 + q^{nd}) 
    \overline{\beta}_{n-1}(xq)
  \begin{bmatrix}
    \uparrow \\ (xq^{n+1})^{d+f} \\ \downarrow \\
    \uparrow \\ (xq^{n+1})^{f} \\ \downarrow 
  \end{bmatrix}
  \begin{array}{cc}
    \Biggr\} & f \leq e \\ 
    & \\
    \Biggr\} & f > e, 
  \end{array}
\end{align}
\begin{align}
\nonumber
  \overline{\beta}_n(x) \left( (xq^{n+1})^d - 1 \right)
  = (q^{-n})^{dk+e} \alpha^e_n(xq) \left( 1 + (xq^{n+1})^d \right), 
\end{align}
\begin{align}
\nonumber 
  & \begin{bmatrix}
    xq^{n+1} & & & -1 \\
    -1 & xq^{n+1} & & \\
    & & \ddots & \\
    & & -1 & xq^{n+1}
  \end{bmatrix}
  \begin{bmatrix}
    \beta^1_n(x) \\ \vdots \\ \beta^d_n(x) 
  \end{bmatrix} \\
\nonumber 
  & = q^{-n(dk+d-1)} \left( 1 + (xq^{n+1})^d \right) 
  \overline{\alpha}_n(xq)
  \begin{bmatrix}
    \uparrow \\ q^{-nf} \\ \downarrow \\
    \uparrow \\ q^{n(d-f)} \\ \downarrow 
  \end{bmatrix}
  \begin{array}{cc}
    \Biggr\} & f \leq e \\ 
    & \\
    \Biggr\} & f > e, 
  \end{array}
\end{align}
\begin{align}
\nonumber
  \overline{\alpha}_n(x) \left( 1 - q^{nd} \right)
  = (xq^{n+1})^{dk+e} \beta^e_{n-1}(xq) \left( 1 + q^{nd} \right).  
\end{align}
The suppressed entries in the matrices are zero.  
Both displayed matrices are invertible, 
and the inverses can be found easily.  
At this point, we gather that the equations are consistent for any $d$.  
After multiplication by the matrix inverses on both sides 
of the first and the third equations, 
it becomes clear that the main line of computations will involve 
$\alpha^e_n(x)$, $\beta^e_n(x)$, $\overline{\alpha}_n(x)$ and $\overline{\beta}_n(x)$.  
$\alpha^f_n(x)$ and $\beta^f_n(x)$'s for $f \neq e$ can be calculated 
using solutions for $\overline{\alpha}_n(x)$ and $\overline{\beta}_n(x)$.  
We then arrive at: 
\begin{align}
\nonumber 
  \alpha^f_n(x) = (-1)^n x^{(dk+e)n} \left( xq^{n+1} \right)^{f-e}
    q^{2(dk+e)\binom{n+1}{2}} 
    \frac{ ( (xq)^{d}; q^{2d} )_\infty ( -q^{d}; q^{d} )_n ( -(xq^{n+1})^{d}; q^{d} )_\infty }
      { ( xq; q^2 )_\infty ( q^{d}; q^{d} )_n ( -(xq^{n+1})^{d}; q^{d} )_\infty }
\end{align}
\begin{align}
\nonumber 
  \times \widetilde{\alpha}^e_0(xq^{2n})
  \begin{cases}
    \left[ \frac{ ( (xq)^{e-f} - (xq)^d ) }{ ( 1 - (xq)^d ) } 
    - q^{nd} \frac{ ( (xq)^d - (xq)^{d+e-f} ) }{ ( 1 - (xq)^d ) }  \right]
      & \textrm{ if } f < e, \\ 
    1 & \textrm{ if } f = e, \\ 
    \left[ \frac{ ( 1 - (xq)^{d+e-f} ) }{ ( 1 - (xq)^d ) } 
    - q^{-nd} \frac{ ( (xq)^{e-f} - 1 ) }{ ( 1 - (xq)^d ) }  \right]
      & \textrm{ if } f > e,
  \end{cases}
\end{align}
\begin{align}
\nonumber 
  \overline{\beta}_n(x) = -(-1)^n x^{(dk+e)n} q^{2(dk+e)\binom{n+1}{2}}
    \frac{ ( (xq^2)^{d}; q^{2d} )_\infty ( -q^{d}; q^{d} )_n ( -(xq^{n+1})^{d}; q^{d} )_\infty }
      { ( xq^2; q^2 )_\infty ( q^{d}; q^{d} )_n ( -(xq^{n+1})^{d}; q^{d} )_\infty }
  \widetilde{\alpha}^e_0(xq^{2n+1}), 
\end{align}
\begin{align}
\nonumber 
  \overline{\alpha}_n(x) = (-1)^n x^{(dk+e)n} q^{2(dk+e)\binom{n+1}{2}}
    \frac{ ( (xq^2)^{d}; q^{2d} )_\infty ( -q^{d}; q^{d} )_n ( -(xq^{n+1})^{d}; q^{d} )_\infty }
      { ( xq^2; q^2 )_\infty ( q^{d}; q^{d} )_n ( -(xq^{n+1})^{d}; q^{d} )_\infty }
  \widetilde{\overline{\alpha}}_0(xq^{2n}), 
\end{align}
\begin{align}
\nonumber 
  \beta^f_n(x) = -(-1)^n x^{(dk+e)n} \left( q^{-n} \right)^{f-e}
    q^{2(dk+e)\binom{n+1}{2}} 
    \frac{ ( (xq)^{d}; q^{2d} )_\infty ( -q^{d}; q^{d} )_n ( -(xq^{n+1})^{d}; q^{d} )_\infty }
      { ( xq; q^2 )_\infty ( q^{d}; q^{d} )_n ( -(xq^{n+1})^{d}; q^{d} )_\infty }
\end{align}
\begin{align}
\nonumber 
  \times \widetilde{\overline{\alpha}}_0(xq^{2n+1})
  \begin{cases}
    \left[ \frac{ ( (xq)^{e-f} - (xq)^d ) }{ ( 1 - (xq)^d ) } 
    - q^{-nd} \frac{ ( 1 - (xq)^{e-f} ) }{ ( 1 - (xq)^d ) }  \right]
      & \textrm{ if } f < e, \\ 
    1 & \textrm{ if } f = e, \\ 
    \left[ \frac{ ( 1 - (xq)^{d+e-f} ) }{ ( 1 - (xq)^d ) } 
    - q^{nd} \frac{ ( (xq)^{d+e-f} - (xq)^d ) }{ ( 1 - (xq)^d ) }  \right]
      & \textrm{ if } f > e.  
  \end{cases}
\end{align}
We made a little twist on the fly, 
and changed $\alpha^e_0(x)$ and $\overline{\alpha}_0(x)$
to $\widetilde{\alpha}^e_0(x)$ and $\widetilde{\overline{\alpha}}_0(x)$, 
respectively.  
This is done to convert some finite products into infinite ones.  
The precise relations could be readily figured out.  

At this point, using these solutions instead of \eqref{defAlphaf}-\eqref{defAlphaBetaBar}
will make Lemmas \ref{lemmaUbarAdjustment} and \ref{lemmaUUbarFunclEqns} work by construction.  
Also, choosing $\widetilde{\alpha}^e_0(x)$ and $\widetilde{\overline{\alpha}}_0(x)$ 
identically 1 will make Lemma \ref{lemmaUUbarInitVals} work.  

\section{Comments and Future Work}
\label{secComm} 

As seen in Section \ref{secConstr}, 
it is mechanical to construct solutions to functional equations 
which are derived from definitions such as Definition \ref{defMain}.  
In other words, the major component in obtaining identities like Theorem \ref{thmOverPtnModd} 
is spotting the partition or overpartition classes.  
In this respect, Sang, Shi and Yee's paper~\cite{SSY-parity-overptn} 
is the primary inspiration for the results presented here.  

Sang, Shi and Yee~\cite{SSY-parity-overptn} 
have evidently positive generating functions, 
or Andrews-Gordon type~\cite{Andrews-PNAS} series,
for their $e = d = 2$ case.  
Even for their complementary case $2e = d = 2$ we do not have evidently positive series here.  
This is a future project.  

Unfortunately, unless $e = d$ or $2e = d$,  
one gets inconsistent functional equations for 
$\widetilde{\alpha}^e_0(x)$ and $\widetilde{\overline{\alpha}}_0(x)$ 
as described in Section \ref{secConstr}, so there are still many missing cases.  
The construction must significantly differ for those cases, 
so it is also left as future work.  

One wonders the possibility of obtaining the main cases (i.e. the $f=e$ cases) of
Theorem \ref{thmOverPtnModd} or Corollary \ref{corollaryOverPtnModdComb} 
as a consequence of Bressoud's all moduli generalization~\cite{Bressoud-RRG-AllModuli} 
of the Rogers-Ramanujan-Gordon identities~\cite{RRG} 
or of Gordon's theorem for overpartitions~\cite{ChenEtAl-RRG-overptn, Lovejoy-RRG-overptn}, 
and then combinatorially obtaining the auxiliary cases (i.e. the $f \neq e$ cases).  

As pointed out in~\cite{KK-AndrewsStyle}, the computations take too long by hand, 
and they have not been even semi-automated yet.  
Computer algebra support for the calculations such as in Section \ref{secConstr} 
will make production of partition or overpartition identities routine, 
since all one has to do is to try various definitions and write 
the corresponding functional equations for the involved generating functions.

\section*{Acknowledgements} 

The results for $d = 2$, along with the missing cases in Theorem \ref{thmSangShiYee}
are part of the second author's PhD dissertation~\cite{MZD-thesis}.  
The construction of series is done in a much more general, 
but computationally heavier, setting there.  

\bibliographystyle{amsplain}

\begin{thebibliography}{10}

\bibitem{Andrews-RRG-analytic}
Andrews, G.E., 
An analytic proof of the Rogers-Ramanujan-Gordon identities. 
\emph{American Journal of Mathematics}, {\bf 88}(4), pp.844--846, 1966

\bibitem{Andrews-PNAS}
Andrews, G.E., 
An analytic generalization of the Rogers-Ramanujan identities for odd moduli. 
\emph{Proceedings of the National Academy of Sciences}, {\bf 71}(10), pp.4082--4085, 1974.  

\bibitem{Andrews-posets-RRG}
Andrews, G.E., 
Partially ordered sets and the Rogers-Ramanujan identities. 
\emph{aequationes mathematicae}, {\bf 12}(1), pp.94--107, 1975.  

\bibitem{Andrews-bailey}
Andrews, G.E., 
\emph{q-Series: Their development and application in analysis, 
number theory, combinatorics, physics, and computer algebra} (Vol. 66). 
American Mathematical Soc., 1986.  

\bibitem{TheBlueBook} G. E. Andrews, 
\textit{The Theory of Partitions}, Addison-Wesley, 1976. 
Reissued, Cambridge University Press, 1998.  

\bibitem{Andrews-parity}
Andrews, G.E., 
Parity in partition identities. 
\emph{The Ramanujan Journal}, {\bf 23}(1), pp.45--90, 2010. 

\bibitem{Bressoud-RRG-AllModuli}
Bressoud, D.M., 
A generalization of the Rogers-Ramanujan identities for all moduli. 
\emph{Journal of Combinatorial Theory, Series A}, {\bf 27}(1), pp.64--68, 1979.  

\bibitem{ChenEtAl-RRG-overptn}
Chen, W.Y., Sang, D.D. and Shi, D.Y., 
The Rogers–Ramanujan–Gordon theorem for overpartitions. 
\emph{Proceedings of the London Mathematical Society}, {\bf 106}(6), pp.1371--1393, 2013. 

\bibitem{CL-overptns}
Corteel, S. and Lovejoy, J., 
Overpartitions. 
\emph{Transactions of the American Mathematical Society}, {\bf 356}(4), pp.1623--1635, 2004.  

\bibitem{GR}
Gasper, G. and Rahman, M., 
\emph{Basic hypergeometric series} (Vol. 96). 
Cambridge university press, 2004. 

\bibitem{RRG}
Gordon, B., 
A combinatorial generalization of the Rogers-Ramanujan identities. 
\emph{American Journal of Mathematics}, {\bf 83}(2), pp.393--399, 1961. 

\bibitem{KK-AndrewsStyle}
Kurşungöz, K., 
Andrews style partition identities. 
\emph{The Ramanujan Journal}, {\bf 36}(1), pp.249--265, 2015.  

\bibitem{Lovejoy-RRG-overptn}
Lovejoy, J., 
Gordon's theorem for overpartitions. 
\emph{Journal of Combinatorial Theory, Series A}, {\bf 103}(2), pp.393--401, 2003. 

\bibitem{RR}
Ramanujan, S. and Rogers, L.J., 
Proof of certain identities in combinatory analysis. 
In \emph{Proc. Cambridge Philos. Soc} {\bf 19}(214-216) p. 3, 1919. 

\bibitem{Rogers}
Rogers, L.J., 
Second memoir on the expansion of certain infinite products. 
\emph{Proceedings of the London Mathematical Society}, {\bf 1}(1), pp.318--343, 1893. 

\bibitem{SSY-parity-overptn}
Sang, D.D., Shi, D.Y. and Yee, A.J., 
Parity considerations in Rogers–Ramanujan–Gordon type overpartitions. 
\emph{Journal of Number Theory}, {\bf 215}, pp.297--320, 2020.  

\bibitem{Schur}
Schur, I., 
\emph{Ein Beitrag zur additiven Zahlentheorie und zur Theorie der Kettenbr\"{u}che}, 1917. 

\bibitem{Selberg}
Selberg, A., 
\"{U}ber einige arithmetische Identit\"{a}ten. 
\emph{Dybwad i Komm.}, 1936.  

\bibitem{MZD-thesis}
ZadehDabbagh, M., 
\emph{Construction of series as generating functions 
and verification type proofs for Rogers-Ramanujan generalizations 
for partitions and overpartitions}, 
PhD thesis, 
Sabanc{\i} University, 2022.  

\end{thebibliography}

\end{document}